\documentclass[10pt]{amsart}
\usepackage{amssymb}

\usepackage{comment}
\usepackage{color}
\usepackage{mathrsfs}
\usepackage{tikz}
\usepackage{amsthm}
\newtheorem{theorem}{Theorem}[section]
\newtheorem{lemma}[theorem]{Lemma}

\newtheorem{claim}[theorem]{Claim}
\newtheorem{corollary}[theorem]{Corollary}

\theoremstyle{definition}
\newtheorem{definition}[theorem]{Definition}

\newtheorem{ques}[theorem]{Question}
\newenvironment{claimproof}[1][Proof]{\begin{proof}[#1]}{\end{proof}}
\newcommand{\w}{{\rm wlim}}

\newtheorem{conjecture}[theorem]{Conjecture}
\theoremstyle{remark}

\numberwithin{equation}{section}

\begin{document}

\title[The unbalanced graph of ergodic transformations]{The unbalanced graph of the conjugacy relation of ergodic measure-preserving transformations}


\author{Bo Peng}
\address{Department of Mathmatics and Statistics, McGill University. 805 Sherbrooke Street West Montreal, Quebec, Canada, H3A 2K6}
\email{bo.peng3@mail.mcgill.ca}


\date{}

\dedicatory{}

\commby{}


\maketitle
\begin{abstract}
    We show that the isomorphism of ergodic measure-preserving transformations is not Borel reducible to the relation induced by the conjugacy action of the full group of an ergodic measure-preserving transformation on itself. This answers a question of Le Ma\^{i}tre in the negative and gives a positive indication towards a conjecture of Sabok. On the other hand, we prove that the conjugacy relation of ergodic measure-preserving transformations is not unbalanced, which answers another question of Le Ma\^{i}tre in the negative. On the positive side, we prove that every two ergodic measure-preserving transformations are connected by a path of length two in the unbalanced graph on the set of ergodic measure-preserving transformations.
\end{abstract}
\section{introduction}

The history of classifying measure-preserving transformations (MPT) up to conjugacy can be traced back to the early 20th century when von Neumann proved that ergodic transformations with discrete spectrum can be classified by the eigenvalues of the associated Koopman operator. In 1971, Ornstein \cite{entropy} proved that Bernoulli shifts are classified by their entropy. Despite many efforts and partial results, the general problem of classifying MPTs was open for quite a long time. 

Let $E$ and $F$ be two equivalence relations on Polish spaces $X$ and $Y$, respectively. A Borel function from $X$ to $Y$ is called a \textbf{Borel reduction} from $E$ to $F$, if for any $x_1,x_2\in X$, we have
$$
x_1Ex_2\,\,\mbox{if and only if}\,\,f(x_1)Ff(x_2).
$$
We say $E$ is \textbf{Borel reducible} to $F$ if there is a Borel reduction from $E$ to $F$ and denote it by $E\leq_B F$. For more details and motivation of a Borel reduction see \cite{Foremanreduction}.  By applying this theory to the classification problems, one can prove impossibility of various types of classification. We call such a result an \textbf{anti-classification theorem}.

In the last 20 years, many types of anti-classification results for isomorphism of MPTs were proved. Foreman and Weiss \cite{Foremanweiss} proved that the conjugacy action of MPTs on ergodic measure-preserving transformations (EMPT) is \textbf{turbulent}. This shows that von Neumann's classification of EMPTs with discrete spectrum is impossible to generalize to classify all EMPTs. A great accomplishment in this direction is that Foreman, Rudolph and Weiss \cite{NonBorelnessEMPT} proved that the isomorphism relation of EMPTs is not Borel. By developing completely new techniques, Foreman and Weiss \cite{SmoothMPTnonBorel} also proved this for smooth ergodic measure-preserving diffeomorphisms of the on the torus, disk and annulus. Recently, Gerber and Kunde \cite{Kakutaninonborel} showed that the Kakutani equivalence for EMPTs is not Borel and Kunde \cite{WeakliymixingnonBorel} proved non-Borelness of weakly mixing diffeomorphisms on the torus, disk and annulus, for both conjugacy and Kakutani equivalence. 

However, the exact complexity of the isomorphism of ergodic measure-preserving transformations remains open. A famous result of Sabok \cite{Choquetcomplete} shows that the affine homeomorphism relation of metrizable compact Choquet simplices is a complete orbit equivalence relation. Since the space of invariant measures of a homeomorphism can be equipped with a Choquet simplex structure, Sabok made the following conjecture (See \cite[Open problem 12]{Surveyforeman} and \cite[Conjecture 4.1]{OPENPROBLEMS}):

\begin{conjecture}
    (Sabok Conjecture) The isomorphism relation of EMPTs is a complete orbit equivalence relation.
\end{conjecture}

However, in the same survey \cite{OPENPROBLEMS}, Le Ma\^{i}tre asked the following question. Let $T\in {\rm Aut}(X,\mu)$ be a measure-preserving transformation. We define its \textbf{full group} $[T]$ as
$$
[T]=\{U\in {\rm Aut}(X,\mu)| \forall x\,\, U(x)\in \{T^n(x)\}_{n\in \mathbb{Z}}\}.
$$
Fix an EMPT, say $T_0$.
\begin{ques} \label{Q0}
    (Le Ma\^{i}tre \cite[Question 4.2]{OPENPROBLEMS}) Does the conjugacy on ${\rm Aut}(X,\mu)$ Borel reduce to the conjugacy relation on $[T_0]$? Does it when we restrict to ergodic transformations?
\end{ques}

Note that Sabok's conjecture implies the negative answer to the above question. The reason is that the uniform metric is a complete bi-invariant metric on $[T_0]$. A Polish group with a complete two-sided compatible invariant metric is called a \textbf{TSI Polish group}. It is known (see \cite{Unblanced}) that any action induced by this kind of group can not be a complete orbit equivalence relation. In this paper we show in Theorem \ref{Answer to Q0} that 
 the conjugacy relation of ergodic measure preserving transformations is not Borel reducible to any TSI group action. This implies a negative answer to question \ref{Q0}.
 
In \cite{Unblanced}, Allison and Panagiotopoulos developed a new theory of \textbf{generic unbalancedness} to prove that an equivalence can not be Borel reducible to any ${\rm TSI}$ group action. See also \cite{LongZheng} for some applications and examples of this theory.

\begin{definition} \label{unbalanceddd}
    Let $G$ be a Polish group acting in a Borel way on a Polish space $X$. Let $x,y\in X$, we write $x\leftrightsquigarrow y$ if for every open neighborhood $V$ of the identity of $G$ and every open set $U\subset X$ having non-empty intersection with the orbit of $x$ or $y$, there exist $g^x,g^y\in G$ with $g^xx\in U$, $g^yy\in U$, so that:
    $$
         g^yy\in \overline{V(g^xx)}\,\,\mbox{and}\,\,(g^xx)\in \overline{V(g^yy)}.
    $$
    
\end{definition}

    Note that if $x \leftrightsquigarrow y$, then $\overline{Gx}=\overline{Gy}$ (see \cite{Unblanced}). Also, for any $y\in X$ sharing the same orbit with $x$, we have $x\leftrightsquigarrow y$.

\begin{definition}
   \cite[Definition 1.1]{Unblanced} Let $X$ be a Polish $G$-space and let $C\subset X$ be a $G$-invariant set. The \textbf{unbalanced graph} associated to the action of $G$ is the graph $(C/G, \leftrightsquigarrow)$ where $C/G=\{Gx\mid x\in C\}$ and $Gx\leftrightsquigarrow Gy$ if and only if $x \leftrightsquigarrow y$.
\end{definition}

\begin{definition}
    \cite{Unblanced} The group action $G$ on $X$ is \textbf{unbalanced} if it has meager orbits and for every $G$-invariant comeager $C\subset X$, there is a comeager $D\subset C$ so that for all $x,y\in D$ there is a path between $Gx$ and $Gy$ in $C/G$. 
\end{definition}

\begin{theorem} \label{unbalance}
    {\rm (Allison and Panagiotopoulos \cite{Unblanced})} If the Polish $G$-space is generically unbalanced, then the orbit equivalence relation is not classifiable by ${\rm TSI}$ Polish group actions. 
\end{theorem}

 In \cite[Question 4.3]{OPENPROBLEMS}, Le Ma\^{i}tre also asked the following question:

\begin{ques} \label{Q1}
    (Le Ma\^{i}tre \cite[Question 4.3]{OPENPROBLEMS}) Is the conjugacy action of ${\rm MPTs}$ on itself unbalanced? Is it when we restricted to ${\rm EMPTs}$? 
\end{ques}

Since we know that every orbit of the conjugacy action of ergodic transformations is dense and meager, Question \ref{Q1} is very natural and we only have to check the unbalanced graph.

By the result of \cite{Unblanced}, a positive answer to Question \ref{Q1} implies the negative answer to Question \ref{Q0}. However, we answer Question \ref{Q1} in the negative.

\begin{theorem} \label{Mainthm}
    The conjugacy actions of ${\rm Aut}(X,\mu)$ on both itself and on ergodic measure-preserving transformations are not unbalanced.
\end{theorem}

On the other hand, the following result seems to be the best positive result, which might be of independent interest.

\begin{theorem}\label{positive} 
    For every pair $T,S$ of ergodic transformations in ${\rm Aut}(X,\mu)$ there is an ergodic transformation $J$ in ${\rm Aut}(X,\mu)$ such that $T\leftrightsquigarrow J \leftrightsquigarrow S$.
\end{theorem}

\section{Preliminaries} \label{pre}

\subsection{Topologies on the set of ergodic transformations}
     All probability spaces considered in this paper are atomless. Given a standard probability space $(X,\mu_X)$, we denote by ${\rm Aut}(X,\mu_X)$ the group of all invertible measure preserving transformations on $(X,\mu_X)$. 
     
      First, we define the \textbf{weak topology} on ${\rm Aut}(X,\mu_X)$. The open neighborhood for this topology of a transformation $T$ is defined as follows: let $A_1,\dots,A_n$ be a collection of disjoint measurable sets in $X$ and $\epsilon>0$,
     $$
   N_{\epsilon}(T,A_1,\dots,A_n)=\{S\in {\rm Aut}(X,\mu_X)| \forall i\leq n\,\,\mu_X(S(A_i)\Delta T(A_i))<\epsilon\}.
     $$
     
     ${\rm Aut}(X,\mu)$ can also be equipped with the \textbf{uniform metric}. The uniform metric $d_{\rm u}$ is defined as:
     $$
    d_{\rm u}(T,S)=\mu_X(\{x| T(x)\neq S(x)\}).
     $$
       ${\rm Aut}(X,\mu_X)$ is not separable thus not Polish under the uniform metric. However, for any ergodic transformation $T_0$, the full group $[T_0]$ is a Polish group under the uniform metric. It is routine to check that $d_{\rm u}$ is two-sided invariant. In a conclusion, $[T_0]$ is a TSI Polish group.

       More details about the topologies on the set of ergodic transformations could be found in \cite[Chapter I]{Kechrisergodic}.
\subsection{Measures, marginals and joinings}

We are going to use the following disintegration theorem:
\begin{theorem} \label{MDT}
    {\rm (Measure Disintegration Theorem, \cite[(17.35)]{Kechrisclassic}) } Let $(X,\mu_X)$ and $(Y,\mu_Y)$ are compact metric spaces equipped with probability measures. Let $\pi: Y \rightarrow X$ be a measurable function. Then there exists a collection of probability measures $\{\mu_x\}_{x\in X}\subset P(Y)$ such that 
    \begin{enumerate}
        \item The function $x \mapsto \mu_x$ is measurable.
        \item $\mu_x$ is a probability measure defined on $\pi^{-1}({x})$.
        \item For any $f\in L^1(Y)$, we have
        $$
      \int_{Y}f(y)d\mu_Y(y)=\int_X(\int_{\pi^{-1}(x)} f(y)d\mu_x(y))d\pi^*\mu(x).
        $$
    \end{enumerate}
\end{theorem}

We use the notation $\pi_{1,2}:X\times X\times X \rightarrow X\times X$ to denote the projection on the first two coordinates.

\begin{corollary} \label{disinte}
    Let $\lambda$ be a measure on $X\times X\times X$. We can find a collection of probability measures $\lambda_{x,y}$ supported on $\{x\}\times \{y\}\times X$ such that for all integrable function $F(x,y,z)$,
    $$
   \int_{X\times X\times X}F(x,y,z)d\lambda(x,y,z)=\int_{X\times X} \int_XF(x,y,z)d\lambda_{x,y}(z)d\pi_{1,2}^*\lambda(x,y).
    $$
\end{corollary}
Let $(X,\mu_X,T)$ be an ergodic dynamical system on a standard probability space and let $(Y,\mu_Y,S)$ be another ergodic transformation on standard probability spaces $Y$.  $X\times Y$ is equipped with natural product Borel structure. For any probability measure $\lambda$ on $X\times Y$, the \textbf{marginal} of $\lambda$ on $X$ is the push forward measure $\pi_X^*\lambda$ where $\pi_X$ denotes the projection onto the $X$ coordinate. 

When $X$ is a compact metric spaces, denote by $P(X)$ the set of probability measures supported on $X$. $P(X)$ is equipped with the weak* topology.

\begin{claim} \label{marginalconverge}
    Let $X$ and $Y$ be compact metric space. Then $\pi_X^*: P(X\times Y)\rightarrow P(X)$ is continuous in the weak* topology. 
\end{claim}

\begin{proof}
    Let $\nu_n$ be a sequence of probability measures supported on $X\times Y$ converging to $\nu$ in the weak* topology. For any real-valued continuous function $f$ on $X$, 
    $$
    {\rm lim}_n\int_{X}f(x)d\pi_X^*\nu_n(x)={\rm lim}_n \int_{X\times Y}f\circ\pi_X(x,y)d\nu_n(x,y)=
    $$
    $$
    =\int_{X\times Y}f\circ \pi_X(x,y)d\nu(x,y)=\int_{X}f(x)d\pi_X^*\nu(x).
    $$
\end{proof}

A \textbf{joining} of $T$ and $S$ is an invariant measure of $T\times S$ on $X\times Y$ with marginal $\mu_X$ and $\mu_Y$ on $X$ and $Y$ respectively. A joining is \textbf{non-trivial} if it is not equal to the product measure $\mu_X\times\mu_Y$. Two transformations are \textbf{disjoint} if they do not have any non-trivial joinings.  

\subsection{Operators}

 Let $\mathcal{H}$ be a Hilbert space and $B(\mathcal{H})$ be the space of bounded operators on $\mathcal{H}$.  $B(\mathcal{H})$ is equipped with the weak operator topology. Precisely, a sequence of operators $P_n$ converges to an operator $P$ if and only if for any $x,y\in \mathcal{H}$,
       $$
       {\rm lim}_n \langle P_nx,y\rangle=\langle Px,y\rangle.
       $$

       We will use the property that the unit ball under the weak operator topology is compact \cite[Theorem 5.1.3]{Operateralgebra}.  
       
       For $(x_n),x \in \mathcal{H}$, we write 
       $$
       \w_k x_n=x
       $$ 
       to denote that the weak limit of $(x_n)$ is equal to $x$.
        
       We  define
$$
L^2_0(X)=\{f\in L^{2}(X)| \int_X fd\mu_X=0\}.
$$
Through out the whole paper, we identify two measurable functions if they agree on a set of measure $1$.

 We define $U_T$ the \textbf{Koopman operator} of $T$ as follows:  $U_T: L^2(X)\rightarrow L^2(X)$, $U_T(f)=f\circ T$. An operator $\Phi:L^2(X)\rightarrow L^2(Y)$ is a \textbf{Markov operator} (see \cite[Page 236]{OTAE}) if 
\begin{enumerate}
    \item $\Phi\geq 0$.
    \item $\Phi(1_X)=1_Y$.
    \item For $f\in L^2(X)$, we have $\int_Xfd\mu_X=\int_Y\Phi(f)d\mu_Y$.
\end{enumerate}

Let $\nu$ be a joining of $T$ and $S$, Then $\nu$ will define the operator $\Phi_{\nu}: L^2(X)\rightarrow L^2(Y)$, such that for any $f\in L^2(X)$ and $g\in L^2(Y)$,
$$
\langle\Phi_{\nu}f,g\rangle= \int_{X\times Y} f(x)g(y)d\nu(x,y).
$$
    Let $\nu=\int_Y \nu_y\times \delta_y d\nu_Y(y)$ be the measure disintegration.

\begin{claim}\label{measuredis}
    For $f\in L^2(X)$, we have $\Phi_{\nu}(f)(y)=\int_Xf(x)d\nu_y(x)$.
\end{claim}
\begin{proof}
    Let $\pi:X\times Y \rightarrow Y$ be the projection on the second coordinate, thus for any $y \in Y$, $\pi^{-1}(y)=X$. By Theorem \ref{MDT}, for any $g\in L^2(Y)$, we have
    $$
    \langle \Phi_{\nu}f,g\rangle=\int_{X\times Y}f(x)g(y)d\nu=\int_{Y}(\int_Xf(x)d\nu_y(x))g(y)d\pi^*\nu(y).
    $$
    Since $\pi^*\nu=\mu_Y$, this shows that $\Phi_{\nu}(f)(y)=\int_Xf(x)d\nu_y(x)$.
\end{proof}

\begin{claim} \label{boundedprese}
    For $f\in L^{\infty}(X)$, we have $||\Phi_{\nu}(f)||_{\infty}\leq ||f||_{\infty}$, in particular, $\Phi_{\nu}(f)\in L^{\infty}(Y)$.  Also, $\Phi_{\nu}$ is a Markov operator.
\end{claim}
\begin{proof}
     The first part of the claim follows from Claim \ref{measuredis}. Alternatively, by the Cauchy--Schwarz inequality and the fact that $\nu$ is a joining, for any $g\in L^2(Y)$ we have $\langle\Phi_\nu(f),g\rangle\leq ||f||_{\infty} ||g||_2$. Take $A=\{y\in Y| \Phi_\nu(f)(y)>||f||_{\infty}\}$, suppose $\mu_Y(A)>0$, we have $\langle\Phi_\nu(f),1_A\rangle>||f||_{\infty}\mu_Y(A)=||f||_{\infty}||g||_2$, a contradiction.
    
    Now we prove that for any $f\in L^2(X)$, we have $
    \int_Xfd\mu_X=\int_Y\Phi_{\nu}f\,d\mu_Y$. By Claim \ref{measuredis} and Theorem \ref{MDT},
    $$
\int_Y\Phi_{\nu}f\,d\mu_Y=\int_{Y}(\int_Xf(x)d\nu_y(x))d\pi^*\nu(y)=\int_{X\times Y} f(x)1_Y(y)d\nu(x,y)=\int_X fd\mu_X.
    $$
    
$\Phi_{\nu}(f)\geq 0$ and $\Phi_{\nu}(1_X)=1_Y$ follows directly from the definition. 
    
\end{proof}

The next claim shows that $\Phi_{\nu}$ is a commuting Markov operator (see \cite[Definition 6.12.]{Glasner})
 \begin{claim} \label{commuting}
     Let $T$ and $S$ be ergodic transformations, $\nu$ is a joining of $T$ and $S$. Then $\Phi_{\nu} \circ U_T=U_S\circ \Phi_{\nu}$.
 \end{claim} 
 \begin{proof}
 Since $\nu$ is an invariant measure of $T\times S$, we have
     $$
 \langle\Phi_{\nu} U_Tf,g\rangle=\int_{X\times Y}f(Tx)g(y)d\nu(x,y)=\int_{X\times Y}f(Tx)g(SS^{-1}y)d\nu(x,y)=
 $$
 $$
 =\int_{X\times Y}f(x)g(S^{-1}y)d\nu(x,y)=\langle\Phi_{\nu} f,U_{S^{-1}}g\rangle=\langle U_S\Phi_{\nu} f,g\rangle.
 $$
 \end{proof}
 
The following claim characterizes when a joining is non-trivial.
\begin{claim} \label{charizenontrivial}
     Let $T\in {\rm Aut}(X,\mu_X)$ and $S\in {\rm Aut}(Y,\mu_Y)$ be ergodic transformations, $\nu$ is a joining of $T$ and $S$. A joining $\nu$ is non-trivial if and only if $\Phi_{\nu}f_0$ is not equal to the constant function $0$ for some $f_0\in L^2_0(X)\cap L^{\infty}(X)$.
 \end{claim}
 \begin{proof}
     ($\Rightarrow$) If $\nu$ is non-trivial, we can find two measurable sets $A,B$ with $0<\mu_X(A)<1$ and $0<\mu_Y(B)<1$ such that such that $\int_{X\times Y}1_A1_Bd\nu\neq \mu_X(A)\mu_Y(B)$. Now take $f_0=1_A-\mu_X(A)$, we have $f_0\neq 0$. Then,
    $$
    \langle \Phi_{\nu} f_0,1_B\rangle=\langle \Phi_{\nu}(1_A-\mu_X(A)),1_B\rangle= \int_{X\times Y} (1_A-\mu_X(A))(x)1_B(y)\,d\nu(x,y) =
    $$
    $$
    =\int_{X\times Y}1_A(x)1_B(y)d\nu(x,y)- \mu_X(A)\mu_Y(B)\neq 0.
    $$
    Thus, we have
    $\langle\Phi_{\nu}(f_0),1_B\rangle\neq 0$ which implies that $\Phi_{\nu}(f_0)\neq 0$.
    
    ($\Leftarrow$) Suppose the joining $\nu$ is the product measure, then for any $f_0\in L_0^2(X)$ and any $g\in L^2(Y)$
    $$
    \langle\Phi_{\nu}(f_0),g\rangle=\int_{X\times Y}f_0(x)g_0(y)d(\mu_X\times\mu_Y)(x,y)=(\int_Xf_0d\mu_X)(\int_Y gd\mu_Y)=0.
    $$
    Thus $\Phi_{\nu}f_0=0$
 \end{proof}
In this paper, we use $\otimes$ to denote the Hilbert space tensor product. We will use the fact that $L^2(X)\otimes L^2(Y)=L^2(X\times Y)$, and for $f\in L^2(X), g\in L^2(Y)$ we will use the notation $f\otimes g$ to denote the function 
$$
f\otimes g(x,y)=f(x)g(y).
$$ 
\subsection{Pairs of joinings}

Now let $(Z,\mu_Z,R)$ be a third ergodic system on a standard probability space $Z$.  Let $\tau$ be a joining of $T$ and $R$ and $\sigma$ be a joining of $S$ and $R$. $\Phi_{\tau}$ and $\Phi_{\sigma}$ define a bilinear form on $L^2(X)\times L^2(Y)$ which extends to a functional $\psi_{\tau,\sigma}$ on $L^2(X)\otimes L^2(Y)=L^2(X\times Y)$ as follows:

$$
\psi_{\tau,\sigma}(f\otimes g)=\langle\Phi_{\tau}(f),\Phi_{\sigma}(g)\rangle=\int_{Z} \Phi_{\tau}(f_0)(z)\Phi_{\sigma}(g_0)(z)d\mu_Z(z).
$$
Using Theorem \ref{MDT}, disintegrate $\tau$ and $\sigma$ over the their coordinate $Z$ as $\{\tau_z\}_{z\in Z}$ and $\{\sigma_z\}_{z\in Z}$.  For any measurable set $E$ on $X\times Y$ define
$$
\theta(E)=\int_Z \tau_z\times\sigma_z(E)d\mu_Z(z).
$$
Then, for any $f\in L^2(X)$ and $g\in L^2(Y)$ by Claim \ref{measuredis} we have 
$$
    \langle\Phi_{\tau} f,\Phi_{\sigma}g\rangle=\int_Z\bigg(\int_X f(x) d\tau_z(x) \int_Y g(y) d\sigma_z(y)\bigg)d \mu_Z(z)=
$$    
$$     =\int_Z \bigg(\int_{X\times Y} f(x)g(y) d(\tau_z(x)\times \sigma_z(y)\bigg)d \mu_Z(z)  = \int_{X\times Y} f(x)g(y)d\theta(x,y).
$$
Thus, 
\begin{equation}\label{joining1}
    \psi_{\tau,\sigma}(f\otimes g)=\int_{X\times Y} f(x)g(y)d\theta(x,y).
\end{equation}


 \begin{claim} \label{relativejoining}
    Let $T\in {\rm Aut}(X,\mu_X)$, $S\in {\rm Aut}(Y,\mu_Y)$ and $R\in {\rm Aut}(Z,\mu_Z)$ be ergodic transformations. Let $\tau$ be a joining of $T$ and $R$ and $\sigma$ be a joining of $S$ and $R$. Let $\theta$ be the probability measure on $X\times Y$ such that 
    $$
    \langle\Phi_{\tau}(f),\Phi_{\sigma}(g)\rangle=\int_{X\times Y} f(x)g(y)d\theta(x,y).
    $$
    Then $\theta$ is a joining of $T$ and $S$.
 \end{claim}
 \begin{proof}
    For any $g\in L^2(Z)$,
     $$
     \int_{X\times Y} 1_X(x)g(y)d\theta(x,y)=\langle\Phi_{\tau}(1_X),\Phi_{\sigma}(g)\rangle=
     $$
     $$
     \langle1_Z,\Phi_{\sigma}g\rangle=\int_{Y\times Z}g(y)1_Z(z)\,d\sigma(y,z)=\int_Yg(y)\,d\mu_Y(y).
     $$
     In other words, the marginal of $\theta$ on $Y$ is $\mu_Y$, similarly, the marginal of $\theta$ on $X$ is $\mu_X$.
     
     For any $f\in L^2(X)$ and $g\in L^2(Y)$, by Claim \ref{commuting}, we have

$$
\int_{X\times Y} f(Tx)g(Sy)d\theta(x,y)=\langle \Phi_{\tau}(U_T\,f),\Phi_{\sigma}(U_S\,g)\rangle=
$$
$$=\langle U_R \Phi_{\tau}f, U_R\Phi_{\sigma}g\rangle=\langle\Phi_{\tau}f,\Phi_{\sigma}g\rangle=\int_{X\times Y} f(x)g(y)d\theta(x,y).
$$

Thus, $\theta$ is $T\times S$ preserving, in other words, $\theta$ is a joining of $T$ and $S$.  
 \end{proof}
 
 Now we define the \textbf{relatively independent joining} $\lambda$ of $\tau$ and $\sigma$. We define $\lambda$ as a probability measure on $X\times Y\times Z$ such that for any bounded measurable functions $f,g,h$ on $(X,\mu_X)$, $(Y,\mu_Y)$ and $(Z,\mu_{Z})$, respectively, we have 
$$
\int_{X\times Y\times Z} f(x)g(y)h(z)d\lambda(x,y,z)=\int_Z\Phi_{\tau}(f)(z)\Phi_{\sigma}(g)(z)h(z)d\mu_Z(z)
$$
We will only use the first part of the following lemma, but we prove the second part to illustrate the definition. 
\begin{lemma} \label{marginalofrel}
    The marginals of $\lambda$ on $X,Y,Z$ are $\mu_X,\mu_Y,\mu_Z$, respectively. The marginals of $\lambda$ on $X\times Z$ and $Y\times Z$ are $\tau$ and $\sigma$ respectively.
\end{lemma}
\begin{proof}
    The fact that marginals on $X,Y,Z$ are $\mu_X, \mu_Y,\mu_Z$ follows directly by Claim \ref{boundedprese} saying that $\Phi_{\nu}$ and $\Phi_{\sigma}$ are Markov operators. 

     Note that by Claim \ref{measuredis} and Claim \ref{boundedprese}, for any bounded functions $f$ on $X $ and $h$ on $Z$, we have
    $$
   \int_{X\times Y\times Z} f(x)1_Y(y)h(z)d\lambda(x,y,z)=
    $$
    $$
    \int_Z\Phi_{\tau}(f)(z)\Phi_{\sigma}(1_Y)(z)h(z)d\mu_Z(z)=\int_Z\Phi_{\tau}(f)(z) h(z)d\mu_Z(z)=
    $$
    $$
    =\int_{X\times Z} f(x)h(z)d\tau(x,z).
    $$
    This implies that the marginal of $\lambda$ on $X\times Z$ is equal to $\tau$. The fact that marginal of $\lambda$ on $Y\times Z$ is $\sigma$ can be proved in the same way.
\end{proof}
\begin{lemma}\label{joining}
    $\lambda$ is invariant under $T\times S\times R$.
\end{lemma}
\begin{proof}
    For bounded functions $f,g,h$ on $(X,\mu_X)$, $(Y,\mu_Y)$ and $(Z,\mu_{Z})$, respectively, by Claim \ref{measuredis}, we have
    $$
    \int_{X\times Y\times Z} f(Tx)g(Sy)h(Rz)d\lambda(x,y,z)=\int_{X\times Y\times Z} U_Tf(x)U_Sg(y)U_Rh(z)d\lambda(x,y,z)=
    $$
    $$
    =\int_Z\Phi_{\tau}U_T(f)(z)\Phi_{\sigma}U_S(g)(z)U_R(h)(z)d\mu_Z(z)=
    $$
    $$
    = \int_ZU_R\Phi_{\tau}(f)(z)U_R\Phi_{\sigma}(g)(z)U_R(h)(z)d\mu_Z(z)=\int_ZU_R(\Phi_{\tau}(f)\cdot\Phi_{\sigma}(g)\cdot h)(z)d\mu_Z(z)
    $$
    $$
    =\int_Z\Phi_{\tau}(f)(z)\Phi_{\sigma}(g)(z)h(z)d\mu_Z(z)=\int_{X\times Y\times Z} f(x)g(y)h(z)d\lambda(x,y,z).
    $$
\end{proof}
\begin{lemma} \label{extendtol2}
    For any $h\in L^2(X)$, and bounded measurable functions $f$ and $g$, we have
    $$
   \int_{X\times Y\times Z} f(x)g(y)h(z)d\lambda(x,y,z)=\int_Z\Phi_{\tau}(f)(z)\Phi_{\sigma}(g)(z)h(z)d\mu_Z(z).
    $$
    \end{lemma}
\begin{proof}
    Since we can write a function $h=h^+-h^-$ as the subtraction of its positive part and negative part, we may assume that $h>0$. Now we take a sequence of simple functions $h_n<h$ such that 
    $$
   {\rm lim}_n\int_Z|h-h_n|d\mu_Z\rightarrow 0.
    $$

    Since $f$ and $g$ are bounded, by Claim \ref{measuredis}, we know that $\Phi_{\tau}(f)$ and $\Phi_{\sigma}(g)$ is also bounded, thus, we have
    $$
    \int_{X\times Y\times Z} f(x)g(y)h(z)d\lambda(x,y,z)= {\rm lim}_n\int_{X\times Y\times Z} f(x)g(y)h_n(z)d\lambda(x,y,z)=
    $$
    $$
    ={\rm lim}_n\int_X\Phi_{\tau}(f)(z)\Phi_{\sigma}(g)(z)h_n(z)d\mu_Z=\int_X\Phi_{\tau}(f)(z)\Phi_{\sigma}(g)(z)h(z)d\mu_Z(z).
    $$
\end{proof}

\section{An answer to Question \ref{Q0}} \label{shortcut}
  Recall that $=^+$ is the equivalence relation defined on all sequences of reals as follows,  for all $(x_n),(y_n)$,
$$
(x_n)=^+(y_n)\,\,\mbox{if}\,\, \exists\,g\in S_{\infty}\,\forall n\,x_n=y_{g(n)}.
$$

We will use the following theorem of Foreman and Louveau \cite{Foremanlouveaubook}. 

\begin{theorem} \label{FLTHM}
    {\rm (Foreman and Louveau \cite[Theorem 65]{Foremanlouveaubook})} $=^+$ is Borel bireducible with the isomorphism  relation of the ergodic transformation with discrete spectrum.
\end{theorem}

    Let $A$ be the set of all countable dense sequences of reals. Note that $A$ is a dense $G_{\delta}$ subset of $\mathbb{R}^{\omega}$.

    We will prove the following theorem. It uses the machinery of unbalancedness of Allison and Panagiotopoulos, it essentially follows the work of \cite{Unblanced}, but we prove it directly.

    \begin{theorem}\label{unb}
         The action of $S_{\infty}$ on $A$ is unbalanced.
    \end{theorem}

  \begin{proof}
  
   First, by the choice of $A$, every orbit is dense. Next, we prove that every orbit is meager. Let $x\in A$ and let 
    $$
    B_x=\{b\in A\mid \,\, \forall m\,b(1)\neq x(m) \}.
    $$
     Note that $B_x=((\mathbb{R}-\{x(m)\}_m) \times\mathbb{R}^{\omega}) \cap A$ which is the intersection of a product of two comeagers and  $A$. Thus, $B_x$ is comeager. Since we remove all $\{x_m\}$ in the first coordinate, we have that $B_x$ is not in the orbit of $x$. Thus, $B_x$ is a subset of the complement of the orbit of $x$ which implies that the orbit of $x$ is meager.
    
    \begin{claim}\label{keylemma}
        For any $x,y\in A$, if $x,y$ share a dense infinite subset, in other words, if there exists two sequences of natural numbers $\{p_n\},\{q_n\}$ such that $\{x_{p_n}\}=\{y_{q_n}\}$ is dense in $\mathbb{R}$, then $x \leftrightsquigarrow y$.
    \end{claim}
    \begin{proof}
        Assume that both $x$ and $y$ contain the dense sequence $\{d_n\}$.
         
        Since every orbit is dense, every open set intersects both orbits of $x$ and $y$, fix an open set $U \subset A$. Take any open neighborhood $V$ of $1_{S_{\infty}}$. We may assume that $V$ is the set of those permutations which fix the first $k$ many natural numbers. Since $\{d_n\}$ is dense, we can find a finite sequence $\{d_{t_1},d_{t_2},...d_{t_r}\}$ with $r>k$ such that  $U'\subset U$ where $U'$ of the form:
        $$
        U'=\{x\in \mathbb{R}^{\omega}|\forall i\leq r,\,x(i)=d_{t_i} \}.
        $$
        Then take $g^x$ and $g^y$ to be two elements in $S_{\infty}$ such that the first $r$-elements of $g^xx$ and $g^yy$ are the same and equal to $(d_{t_1},\dots,d_{t_r})$. So we have $g^xx,g^yy\in U'$ thus in $U$. Now since $V$ fixes the first $k$ many natural numbers and $k<r$, we have that $\overline{Vg^xx}$ is the set of all sequences whose first $k$ many coordinates are equal to $(d_{t_1},d_{t_2},\dots,d_{t_k})$, similarly for $\overline{Vg^yy}$.  This implies that $\overline{Vg^xx}=\overline{Vg^yy}$. In particular, we have that $x \leftrightsquigarrow y$.
    \end{proof}
   
   Let $P$ be defined as
   $$
   P=\{(x_n)\in A| \,(x_{2n})\,\,\mbox{and}\, (x_{2n+1})\,\,\mbox{are dense in}\,\, \mathbb{R}\}.
   $$
   Note that $P$ is dense $G_{\delta}$ in $A$. We view $P$ as a subspace of $\mathbb{R}^{\omega}\times \mathbb{R}^{\omega}$ by decomposing into even coordinates and odd coordinates. For $x,y\in \mathbb{R}^{\omega}$, let $f: \mathbb{R}^{\omega}\times \mathbb{R}^{\omega}\rightarrow \mathbb{R}^{\omega}$ be the map such that $f(x,y)(2n)=x(n)$ and $f(x,y)(2n+1)=y(n)$. 
   
   Now we finish the proof of unbalancedness. For any $S_{\infty}$-invariant comeager subset $C$ of $A$, we know that $P\cap C$ is comeager. And since $P\cap C$ is a subset of $W$, we can view $P\cap C$ as a subspace of $\mathbb{R}^{\omega}\times \mathbb{R}^{\omega}$.   By the Kuratowski-Ulam theorem, there is a comeager set $D$ in $\mathbb{R}^{\omega}$ such that for every $x\in D$, the section of $x$ in $P\cap C$ is a comeager in $\mathbb{R}^{\omega}$. In other words, for all $x\in D$
   $$
   (P\cap C)_x =\{y\in \mathbb{R}^{\omega}| f(x,y)\in P\cap C\}.
   $$
   is comeager in $\mathbb{R}^{\omega}$, thus in $C$. We are going to show that $D$ witnesses the definition of unbalancedness. For any $x,y\in D$, by the choice of $C$, $(P\cap C)_x$ and $(P\cap C)_y$ are comeager in $C$, so their intersection is generic, in particular, non-empty. We take any $z$ from $(P\cap C)_x \cap (P\cap C)_y$. We have $f(x,z)\in C$ and $f(y,z)\in C$. By Claim \ref{keylemma}, 
   $$
   x \leftrightsquigarrow f(x,z) \leftrightsquigarrow f(y,z) \leftrightsquigarrow y. 
   $$
   which proves the theorem.
\end{proof}

Below we give an answer of Question \ref{Q0}. By Theorem \ref{FLTHM}, it is enough to show $=^+$ is not Borel reducible to any ${\rm TSI}$ Polish group action since the group $[T_0]$ is a TSI group. In fact, we give two proofs. The first one uses Theorem \ref{unb}. 
The second proof uses pinned equivalence relation.

   \begin{theorem}\label{Answer to Q0}
       The conjugacy relation of ergodic measure preserving transformations is not Borel reducible to any TSI group action.
   \end{theorem}
   \begin{proof}

    We give two proofs of the theorem. 
    
    Since $A$ is $S_{\infty}$-invariant, we know that $E^A_{S_{\infty}}$ is Borel reducible to $=^+$. Together with Theorem \ref{FLTHM}, we know that we can reduce an unbalanced action to the conjugacy relation of ergodic measure preserving transformations which implies the theorem.


Another proof uses pinned equivalence relations.  $=^+$ is unpinned \cite[Claim 17.2.1]{BorelEQ}, ${\rm CLI}$ group actions are all pinned \cite[Chapter 17.4]{BorelEQ} and ${\rm TSI}$-groups are all ${\rm CLI}$. Thus, $=^+$ is not reducible to any ${\rm TSI}$ group actions.

\end{proof}
\section{There is no $\leftrightsquigarrow$ between disjoint transformations} \label{nosquigarrow}

In this chapter, we will use the standard probability space $(X,\mu)$ and additionally assume that $X$ is a compact metric space. 

\begin{lemma} \label{marginalstable}
    Let $(\nu_n)$ be a sequence of probability measures on $X\times X$. Assume that for any $n\in \mathbb{N}$, the marginals of $\nu_n$ on both coordinates are $\mu$. Suppose $\nu_n$ converges to $\nu$ in the weak* topology. Then for any measurable sets $C,D\subset X$, we have $\nu(C\times D)={\rm lim}_n \nu_n(C\times D)$.
\end{lemma}

\begin{proof}
    First,  $\nu$ also has marginal $\mu$ on both coordinates since mapping the measure to its marginals is continuous in the weak* topology by Claim \ref{marginalconverge}. 
    
    Since continuous functions are dense in $L^1(X,\mu)$, for any $\epsilon>0$, we can find continuous functions $|f|\leq1$ and $|g|\leq1$ such that 
    $$
    \int_X |f-\chi_C| d\mu <\epsilon,
    $$
    $$
     \int_X|g-\chi_D| d\mu <\epsilon.
    $$

    Now for any measure $\tau$ on $X\times X$ with marginal $\mu$ on the both coordinates, we have 
    $$
    \int_{X\times X} |\chi_C(x) \chi_D(y)-f(x)g(y)|d\tau \leq
    $$
    $$
    \leq \int_{X\times X} |\chi_C(x)-f(x)| |\chi_D(y)|d\tau+\int_{X\times X}|f(x)||g(y)-\chi_D(y)|d\tau \leq 
    $$
    \begin{equation} \label{product}
         \leq \int_X|\chi_C-f|d\mu+\int_X|g-\chi_D|d\mu<2\epsilon.
    \end{equation}
   
    By (\ref{product}), we know
    $$
    |\nu_n(C\times D)-\nu(C\times D)|=|\int_{X\times X}\chi_C(x)\chi_D(y) d\nu_n-\int_{X\times X}\chi_C(x)\chi_D(y) d\nu|\leq
    $$ 
    $$
    |\int_{X\times X} (\chi_C (x)\chi_D(y)-f(x)g(y))d\nu_n|+|\int_{X\times X} f(x)g(y)d\nu_n-\int_{X\times X} f(x)g(y) d\nu|+
    $$
    $$
    +|\int_{X\times X} (\chi_C(x)\chi_D(y)-f(x)g(y))\, d\nu|\leq
    $$
    $$
    \leq 4\epsilon+|\int_{X\times X} f(x)g(y)d\nu_n-\int_{X\times X} f(x)g(y) d\nu|.
    $$

    Since $fg$ is continuous, we know that $|\int_{X\times X} fgd\nu_n-\int_{X\times X} fg d\nu|$ converges to $0$ as $n$ goes to infinity. Also, since $\epsilon$ was chosen first and could be chosen arbitrarily small, this proves the lemma.
\end{proof}
Now take any set $A$ with $\mu(A)=\frac{1}{2}$. Let 
    $$
    V=\{H \in {\rm Aut}(X,\mu)| \mu(H^{-1}(A)\Delta A)<\frac{1}{8}\}.$$ 
    
    In the Lemma below, the closure stands for the closure in the weak topology.
\begin{lemma}\label{SnotiT}
    
        If $T$ and $S$ are disjoint transformations on $X$, then we have 
        $$S \not \in \overline{\{HTH^{-1}| H\in V\}}.
        $$

    \begin{proof}
        We prove by contradiction. Suppose we can find $H_n \in V$ such that $H_nTH_n^{-1}$ weakly converges to $S$.

        Let $\nu_n$ be the probability measure defined on $X\times X$ such that for measurable sets $C,D\subset X$, we have 
        $$
      \nu_n(C\times D)= \mu(C\cap H^{-1}_n(D)).
        $$
        Note that the marginals of $\nu_n$ are $\mu$ on both coordinates since $H_n$ is measure preserving.
        
        By the compactness of weak* topology on $P(X)$, we can find a convergent subsequence of $\nu_n$ and write $\nu$ for its limit. Without loss of generality, we assume that $\nu_n$ converges to $\nu$ in the weak* topology. By Claim \ref{marginalconverge}, we know that the marginals of $\nu$ are $\mu$ on both coordinates.

        We first show that $\nu$ is an invariant measure for $T\times S$. For any measurable sets $C,D \subset X$, by Lemma \ref{marginalstable} we have
        \begin{equation}\label{wn1}
            \nu(T^{-1}(C)\times S^{-1}(D))={\rm lim}_n \nu_n(T^{-1}(C)\times S^{-1}(D))=
         \end{equation} 
         $$
         ={\rm lim}_n\mu(T^{-1}(C)\cap H_n^{-1}S^{-1}(D))={\rm lim}_n\mu(H_nT^{-1}(C)\cap S^{-1}(D)).
        $$
        Also,
        \begin{equation} \label{vn1}
            \nu(C\times D)={\rm lim}_n\nu_n(C\times D)=
        \end{equation}
       $$
       ={\rm lim}_n\mu(C\cap H_n^{-1}(D))={\rm lim}_n\mu(H_nT^{-1}(C)\cap H_nT^{-1}H_n^{-1}(D)).
        $$
        
        Let $W_n=H_nT^{-1}(C)\cap S^{-1}(D)$ and $V_n=H_nT^{-1}(C)\cap (H_nT^{-1}H_n^{-1}(D))$.  By (\ref{wn1}) and (\ref{vn1}), we have 
        \begin{equation} \label{wn}
            \nu(T^{-1}(C)\times S^{-1}(D))={\rm lim}_n \mu(W_n).
        \end{equation}
        and
        \begin{equation} \label{vn}
            \nu(C\times D)= {\rm lim}_n \mu(V_n).
        \end{equation}
        Observe that
        \begin{equation} \label{wnvn}
            W_n\Delta V_n=H_nT^{-1}(C) \cap[S^{-1}(D)\Delta H_nT^{-1}H_n^{-1}(D)].
        \end{equation}
        Since $H_nTH_n^{-1}$ weakly converges to $S$, we have that $H_nT^{-1}H_n^{-1}$ weakly converges to $S^{-1}$. By (\ref{wnvn}), we know
        $$
      {\rm lim}_n\mu(W_n\Delta V_n )=
      $$
      $$
      ={\rm lim}_n\mu(H_nT^{-1}(C)\cap [S^{-1}(D)\Delta H_nT^{-1}H_n^{-1}(D)])\leq  
      $$
      $$
      \leq {\rm lim}_n \mu(S^{-1}(D)\Delta H_nT^{-1}H_n^{-1}(D))=0.
        $$
       In  conclusion, we have
        \begin{equation} \label{samelimit}
            {\rm lim}_n\mu(W_n)={\rm lim}_n\mu(V_n).
        \end{equation}
        
        By (\ref{samelimit}) together with (\ref{wn}) and (\ref{vn}), we have
        $$
       \nu(T^{-1}(C)\times S^{-1}(D))={\rm lim}_n\mu(W_n)={\rm lim}_n\mu(V_n)=\nu(C\times D).
        $$
        The above implies that $\nu$ is a joining of $T$ and $S$.

        Since $T$ and $S$ are disjoint, the only possible joining is $\mu \times \mu$. Thus, we have $\nu=\mu \times \mu$. Since $H_n\in V$, we know that
        $$
        \nu_n(A\times A)=\mu(A\cap H_n^{-1}(A))\geq \frac{3}{8}.
        $$
        By Lemma \ref{marginalstable}, we have $\nu(A\times A)={\rm lim}_n\nu_n(A\times A)\geq \frac{3}{8}$, however, $(\mu\times \mu)(A\times A)=(\mu(A))^2=\frac{1}{4}$, a contradiction.

    \end{proof}

\end{lemma}
\begin{theorem} \label{disjoint}
    Let $T$ and $S$ be two ergodic transformations on $X$. If $T$ is disjoint with $S$, then $T \not\leftrightsquigarrow S$.
\end{theorem}

\begin{proof}
    Suppose $T \leftrightsquigarrow S$, we can find $T'$ isomorphic to $T$ and $S'$ isomorphic to $S$ such that $S'\in \overline{\{HT'H^{-1}| H\in V\}}$, but since $T$ and $S$ are disjoint, we know $T'$ and $S'$ are also disjoint, this contradicts Lemma \ref{SnotiT}.
\end{proof}

\section{Generic transformations can not be non-disjoint with two disjoint transformations} \label{genericnot}

The following definition shows up under various names in the literature. We are going to use the terminology from \cite{Glasner}. 

\begin{definition}
     (\cite[Definition 7.8]{Glasner}) For $\alpha\in [0,1]$, a transformation $R\in {\rm Aut}(X,\mu)$ is called $\alpha$-\textbf{weak mixing} if there exists an increasing sequence of natural numbers $(n_k)$ such that for all measurable sets $A$ and $B$, 
$$
{\rm lim}_k\mu(A\cap R^{n_k}(B))= (1-\alpha) \mu(A\cap B)+\alpha\mu(A)\mu(B) .
$$
\end{definition}

The following claims are well-known.

\begin{claim} \label{ergodicity}
    If a transformation $R$ is $\alpha$-weak mixing for $0<\alpha\leq 1$, then $R$ must be ergodic.
\end{claim}
\begin{proof}
    Let $A$ be an $R$-invariant subset. Take $B=A$, by the definition of $\alpha$-weak mixing, we have
    $$
    \mu(A)={\rm lim}_k\mu(A\cap R^{n_k}(A))=(1-\alpha)\mu(A)+\alpha\mu^2(A).
    $$
    Since $\alpha\neq 0$, we have $\mu(A)=1$ or $0$.
\end{proof}
\begin{claim} \label{seconddefi}
    A transformation $R$ is $\alpha$-weak mixing if and only  if there exists a sequence $(n_k)$ such that for any $f,g\in L^2(X)$,
    $$
    {\rm lim}_k\langle U_R^{n_k}f,g\rangle= (1-\alpha) \langle f,g \rangle+ \alpha(\int_Xfd\mu)(\int_X gd\mu).
    $$
\end{claim}

\begin{proof}
    ($\Rightarrow$) First, for any measurable sets $A$ and $B$, we have 
    $$
    {\rm lim}_k \int_X 1_A1_{R^{n_k}(B)}(x)d\mu={\rm lim}_k\mu(A\cap R^{n_k}(B))= (1-\alpha)\mu(A\cap B)+\alpha \mu(A)\mu(B)=
    $$
    $$
    =(1-\alpha)\int_X1_A1_Bd\mu+\alpha \int_X 1_Ad\mu \int_X 1_Bd\mu.
    $$
    Thus, we have ${\rm lim}_k\langle U_R^{n_k}1_A,1_B\rangle= (1-\alpha) \langle 1_A,1_B \rangle+ \alpha(\int_X 1_Ad\mu)(\int_X 1_Bd\mu)$. Since the Koopman operator is linear, we know the equality holds for simple functions. Then for any $f,g\in L^2(X)$, we can approximate $f,g$ by simple functions.
    
    ($\Leftarrow$) This can be done by taking $f$ and $g$ to be $1_A$ and $1_B$ for any measurable sets $A,B$.
\end{proof}

The following two corollaries directly follow from Claim \ref{seconddefi}.
\begin{corollary} \label{operator}
A transformation $R$ is $\alpha$-weak mixing if and only  if there exists a sequence $(n_k)$ such that for any $f\in L^2(X)$, we have
$U_R^{n_k}f$ weakly converges to $(1-\alpha)f+\alpha(\int_Xf d\mu)1_X$. Equivalently, $U_R^{n_k}$ converges in the weak operator topology to $(1-\alpha)I+\alpha P_1$, where $P_1$ is the projection to the subspace constant functions and $I$ is the identity operator.
\end{corollary}

\begin{corollary} \label{l20}
    Let $R$ be an $\alpha$-weak mixing transformation. Then there exists a sequence $(n_k)$ such that for all $f,g\in L^2_0(X)$, we have $\w_kU_R^{n_k}f(1-\alpha)f=U_R^{n_k}f$.
\end{corollary}

We will use the following theorem of Stepin. In \cite{genericmixing}, Stepin uses the name $\alpha$-\textit{mixing} for $\alpha$-weak mixing  transformations, but the definition is exactly the one given by Corollary \ref{operator}.
\begin{theorem} \label{stepin}
    {\rm (Stepin, \cite[Theorem 2]{genericmixing})} For any $\alpha\in [0,1]$, the set of $\alpha$-weak mixing transformations is generic in ${\rm Aut}(X,\mu)$.
\end{theorem}

In \cite{nonjointfactor}, Friedman used the terminology \textit{condition} ($\alpha$) to denote $\alpha$-weak mixing transformation and proved that $\alpha$-weak mixing transformations can not have independent factors, if $\alpha\in (0,1)$. This was also mentioned in \cite{genericproperty}. In other words, an $\alpha$-weak mixing transformation $R$ can not factor onto $T\times S$ where $T$ and $S$ are measure preserving transformations if $\alpha\in (0,1)$. Note that for disjoint transformations $T$ and $S$, a transformation $R$ factors onto both $T$ and $S$ if and only if it factors on to the product $T\times S$. This motivates the following result.

\begin{theorem} \label{genericdisjoint}
    Let $T$ and $S$ be two disjoint ergodic transformations, $R$ be an $\alpha$-weak mixing transformation with $\alpha\in (0,1)$. Then $R$ must be disjoint with $T$ or $S$.
\end{theorem}

\begin{proof}
    We argue by contradiction. Suppose we can find ergodic transformations $T,S$ and $R$ in ${\rm Aut}(X,\mu)$ such that $T$ and $S$ are disjoint, $R$ is $\alpha$-weak mixing and is not disjoint with each of $T$ and $S$. We fix an increasing sequence $(n_k)$ of natural numbers
    witnessing that $R$ is $\alpha$-weak mixing.
    
    Let $\tau$ be a non-trivial joining of $T$ and $R$ and $\sigma$ be a nontrivial joining of $S$ and $R$. Recall that we denote by $\Phi_{\tau}$ and $\Phi_{\sigma}$ be the commuting Markov operators of $\tau$ and $\sigma$, respectively.  For functions $f,g\in L^2(X)$, denote by $f\cdot g$ the function 
    $$
    f\cdot g(x)=f(x)g(x)
    $$
    
    By Claim \ref{commuting}, we have,

    \begin{equation} \label{R and T}
         U_R\circ \Phi_{\tau}= \Phi_{\tau}\circ U_T,
    \end{equation}
   
    and
    \begin{equation} \label{R and S}
         U_R \circ \Phi_{\sigma}= \Phi_{\sigma}\circ U_S,
    \end{equation}

By the compactness of the unit ball in $B(L^2(X))$ under weak operator topology \cite[Theorem 5.1.3]{Operateralgebra}, we can find a subsequence $(m_k)$ of $(n_k)$ such that $U_T^{m_k}$ and $U_S^{m_k}$ converges to $P_T$ and $P_S$ respectively. Without loss of generality, we may assume that $(m_k)=(n_k)$.
   \begin{claim}\label{operatortimes}
      For any $f,g\in L^2_0(X)$, we have 
      $$
      (1-\alpha)\Phi_{\tau}(f)=\Phi_{\tau} P_T(f)$$
      and 
      $$(1-\alpha) \Phi_{\sigma}(g)=\Phi_{\sigma} P_S(g).$$
   \end{claim}
   \begin{claimproof}
       By Corollary \ref{l20} and (\ref{R and T}), we have
      $$
        (1-\alpha)\Phi_{\tau}(f)=\w_kU_R^{n_k}\Phi_{\tau}(f)=\w_k \Phi_{\tau} U^{n_k}_T(f)=\Phi_{\tau} P_T(f).
    $$
        Similarly, by Corollary \ref{l20} and (\ref{R and S}), we have 
      $$
        (1-\alpha) \Phi_{\sigma}(g)=\w_kU_R^{n_k}\Phi_{\sigma}(g)=\w_k\Phi_{\sigma}U_S^{n_k}(g)=\Phi_{\sigma} P_S(g).
     $$ \end{claimproof}

\begin{lemma} \label{continue}
Let $f_n\in L^2(X)$ be a sequence of bounded functions such $||f_n||_{\infty}$ is uniformly bounded and $f_n$ weakly converges to $f$ in $L^2(X)$. Let $g_n\in L^2(X)$ be a sequence of bounded functions such that $||g_n||_{\infty}$ is uniformly bounded and $g_n$ weakly converges to $g$ in $L^{2}(X)$. Then $\Phi_{\tau}(f_n)\cdot \Phi_{\sigma}(g_n)$ weakly converges to $\Phi_{\tau}(f)\cdot\Phi_{\sigma}(g)$ in $L^2(X)$.
\end{lemma}

\begin{proof}
    Denote by $\lambda$ the relatively independent joining of $\tau$ and $\sigma$. Let $\lambda_{x,y}$ be the measure disintegration of $\lambda$ with respect to the map to the projection $\pi_{1,2}:X\times X\times X \rightarrow X\times X$ onto the first two coordinates. 
    
    We fix $h\in L^2(X)$. By Lemma \ref{extendtol2}, we have
    $$
   \int_X\Phi_{\tau}(f_n)(z)\Phi_{\sigma}(g_n)(z)h(z)d\mu(z)=\int_{X\times X\times X}f_n(x)g_n(y)h(z)d\lambda(x,y,z).
    $$
    Now by Corollary \ref{disinte}, we have
    $$
    \int_{X\times X\times X}f_n(x)g_n(y)h(z)d\lambda(x,y,z)=
    $$
    $$
    =\int_{X\times X}(\int_Xf_n(x)g_n(y)h(z)d\lambda_{x,y}(z))d\pi_{1,2}^*\lambda(x,y)=
    $$
    $$
    =\int_{X\times X}f_n(x)g_n(y)(\int_Xh(z)d\lambda_{x,y}(z)) d\pi_{1,2}^*\lambda(x,y).  
    $$
    By Lemma \ref{relativejoining} $\lambda$ is invariant under $T\times S\times R$, so $\pi^*_{1,2}\lambda$ is  invariant under $T\times S$. Since $T$ and $S$ are disjoint, by Claim \ref{marginalofrel}, we have $\pi^*_{1,2}\lambda=\mu\times \mu$.

    Now we put 
    $$
    H(x,y)=\int_Xh(z)d\lambda_{x,y}(z).
    $$ 

    \begin{claim}\label{inl2}
        $H\in L^2(X\times X,\mu\times \mu)$.
    \end{claim}
    \begin{claimproof}
       First by Cauchy--Schwarz, for any $(x,y)\in X\times X$ we have
       $$
       |H(x,y)|^2=|\int_Xh(z)d\lambda_{x,y}(z)|^2\leq\int_X|h(z)|^2d\lambda_{x,y}(z).
       $$
       Thus, by the fact that $\pi_{1,2}^*\lambda=\mu\times \mu$, Claim \ref{marginalofrel} and Corollary \ref{disinte} we have
        $$
        ||H||^2_{2}\leq \int_{X\times X}\int_X|h(z)|^2d\lambda_{x,y}(z)d(\mu\times \mu)(x,y)=
        $$
        $$
        =\int_{X\times X}\int_X|h(z)|^2d\lambda_{x,y}(z)d\pi_{1,2}^*\lambda(x,y)
        =\int_{X\times X\times X}|h(z)|^2d\lambda(x,y,z)=||h||^2_{2}.
        $$
        The lemma follows by the fact that $h\in L^2(X)$.
    \end{claimproof}
    \begin{claim} \label{compute}
        For any $f,g\in L^{\infty}(X)$, we have
        $$
       \langle f\otimes g,H\rangle= \int_X\Phi_{\tau}(f_n)(z)\Phi_{\sigma}(g_n)(z)h(z)d\mu(z).
        $$
    \end{claim}
    \begin{claimproof}
        By Corollary \ref{disinte} and the fact that $\pi_{1,2}^*\lambda(x,y)=(\mu\times \mu)(x,y)$, we have
        $$
     \int_X\Phi_{\tau}(f_n)(z)\Phi_{\sigma}(g_n)(z)h(z)d\mu(z)=
     $$
     $$
     =\int_{X\times X}f(x)g(y)(\int_Xh(z)d\lambda_{x,y}(z)) d\pi_{1,2}^*\lambda(x,y)
     $$
     $$
     =\int_{X\times X}f(x)g(y)(\int_Xh(z)d\lambda_{x,y}(z)) d(\mu\times \mu)(x,y)=\langle f\otimes g, H\rangle
     $$
    \end{claimproof}
    \begin{claim} \label{tensor}
        $f_n\otimes g_n$ weakly converges to $f\otimes g$ in $L^2(X\times X, \mu\times \mu)$.
    \end{claim}
    \begin{claimproof}
        We just have to prove that $\langle f_n\otimes g_n,h\otimes k\rangle$ converges to $\langle f\otimes g,h\otimes k\rangle$. The lemma the will follow from the fact that simple tensors span a dense subspace in $L^2(X\times X)$.
        However, $\langle f_n\otimes g_n,h\otimes k\rangle=\langle f_n,h\rangle \langle g_n,k \rangle$ which converges to $\langle f\otimes g,h\otimes k\rangle$.
    \end{claimproof}

    Now we can finish the proof, by Claim \ref{tensor} and Claim \ref{inl2}, we have
    $\langle f_n\otimes g_n,H\rangle$ converges to $\langle f\otimes g,H\rangle$. By Claim \ref{compute}, we have
    $$
{\rm lim}_n\int_X\Phi_{\tau}(f_n)\Phi_{\sigma}(g_n)hd\mu ={\rm lim}_n \langle f_n\otimes g_n,H\rangle=
$$
$$
=\langle f\otimes g,H\rangle
    = \int_X\Phi_{\tau}(f)(z)\Phi_{\sigma}(g)(z)h(z)d\mu
    $$
    which finishes the proof.
\end{proof}
    
    \begin{lemma} \label{0preserved}
        For any $f,g\in L^2_0(X)\cap L^{\infty}(X)$, we have $\Phi_{\tau}(f)\cdot \Phi_{\sigma} (g)\in L^2_0(X)$.
    \end{lemma}
    \begin{proof}
        Since $f,g\in L^{\infty}(X)$, by Claim \ref{boundedprese}, we know that both $\Phi_{\tau}(f)$ and $ \Phi_{\sigma} (g)$ are in $L^{\infty}(X)$. Thus, we have $\Phi_{\tau}(f)\cdot \Phi_{\sigma} (g)\in L^2(X)$.
       Let $\theta$ be the measure defined in Section 2.4, by (\ref{joining1}), it has the property that 
        $$
        \int_{X\times X}f(x)g(y)d\theta(x,y)=\int_X \Phi_{\tau}(f)(x)\Phi_{\sigma}(g)(x)d\mu(x).
        $$
        We claim that $\theta$ is a joining of $T$ and $S$. This follows either directly from Claim \ref{relativejoining} or from  Claim \ref{marginalofrel} together with the fact that $\theta=\pi^*_{1,2}\lambda$, which follows directly from the definition of $\lambda$.
        
        Since $T$ and $S$ are disjoint, we know that $\theta=\mu\times \mu$, thus
        $$
    \int_X \Phi_{\tau}(f)\cdot \Phi_{\sigma} (g)d\mu=\int_{X\times X}f(x)g(y)d(\mu\times \mu)(x,y)=(\int_Xfd\mu)(\int_Xgd\mu)=0.
        $$
    \end{proof}
  
    \begin{claim} \label{equality}
        For any $f,g \in L^2_0(X)\cap L^{\infty}(X)$, we have
        $$
        (1-\alpha) \Phi_{\tau}(f)\cdot\Phi_{\sigma}(g)=\Phi_{\tau}P_Tf\cdot\Phi_{\sigma}P_Sg.
        $$
    \end{claim}
    \begin{proof}
        We let $f_k=U_T^{n_k}f$ and $g_k=U_S^{n_k}g$. Note that $||f_k||_{\infty}=||f||_{\infty}$ and $||g_k||_{\infty}=||g||_{\infty}$. By our assumption, we have $\w_k f_k=P_Tf$ and $\w_k g_k=P_Sg$. By (\ref{R and T}) and (\ref{R and S}), we have that for any $k$,
        $$
        U_R^{n_k} (\Phi_{\tau}(f)\cdot \Phi_{\sigma}(g))= (U_R^{n_k}\Phi_{\tau}f)\cdot (U_R^{n_k}\Phi_{\sigma}g)=(\Phi_{\tau}U_T^{n_k}f) \cdot (\Phi_{\sigma}U_S^{n_k}g)=\Phi_{\tau}f_k \cdot\Phi_{\sigma}g_k.
        $$
        Now we take the weak limit of both sides in $L^2(X)$. By Corollary \ref{l20} and Lemma \ref{0preserved}, on the left side, we have $\w_kU_R^{n_k} (\Phi_{\tau}(f)\cdot \Phi_{\sigma}(g))=(1-\alpha)\Phi_{\tau}(f)\cdot \Phi_{\sigma}(g)$.  By Lemma \ref{continue}, on the right side, we have $\w_k\Phi_{\tau}f_k \cdot\Phi_{\sigma}g_k=(\Phi_{\tau}P_Tf)\cdot (\Phi_{\sigma}P_Sg)$.
    \end{proof} 
    \begin{lemma} \label{non-disjoint}
        There exist $f_0,g_0\in L^2_0(X)\cap L^{\infty}(X)$ such that $\Phi_{\tau}(f_0)\cdot\Phi_{\sigma}(g_0)\neq 0$.
    \end{lemma}
    \begin{proof}
        By Claim \ref{charizenontrivial}, since $\tau$ is a non-trivial joining, we can find $f'_0\in L^2_0(X)\cap L^{\infty}(X)$ such that $\Phi_{\tau}(f'_0)\neq 0$. Let $A=\{x\in X|\,\Phi_{\tau}(f'_0)(x)\neq 0\}$, we know that $\mu(A)>0$. Similarly, we can find $g_0\in L^2_0(X)\cap L^{\infty}(X)$ such that $\Phi_{\sigma}(g_0)\neq 0$. Let $B=\{x\in X|\,\Phi_{\sigma}(g_0)(x)\neq 0\}$, we know that $\mu(B)> 0$. Now by Claim \ref{ergodicity}, we know that $R$ is ergodic, so we can find $n\in \mathbb{N}$ such that $\mu(R^{-n}(A)\cap B)>0$. Let $f_0=U_T^nf'_0$. Note that $f_0\in L^2_0(X)\cap L^{\infty}(X)$. Also, by (\ref{R and T}), we have
        $$
        U_R^n\Phi_{\tau}(f'_0)=\Phi_{\tau} U^n_T(f'_0)=\Phi_{\tau}(f_0).
        $$
        which implies that the support of $\Phi_{\tau}(f_0)$ is $R^{-n}(A)$. The latter set has positive intersection with the support of $\Phi_{\sigma}(g_0)$, thus we have $\Phi_{\tau}(f_0)\cdot \Phi_{\sigma}(g_0)\neq 0$.
    \end{proof}
    Now by Lemma \ref{non-disjoint}, fix $f_0,g_0\in L^2_0(X)\cap L^{\infty}(X)$ such that $\Phi_{\tau}(f_0)\cdot\Phi_{\sigma}(g_0)\neq 0$.
     Now by Claim \ref{equality} and Claim \ref{operatortimes}, we have
    $$
    (1-\alpha) \Phi_{\tau}(f_0)\cdot\Phi_{\sigma}(g_0)=(\Phi_{\tau}P_Tf_0)\cdot(\Phi_{\sigma}P_Sg_0)=
    $$
    $$
    =(1-\alpha)\Phi_{\tau}(f_0)\cdot(1-\alpha)\Phi_{\sigma}(g_0)=(1-\alpha)^2\Phi_{\tau}(f_0)\cdot\Phi_{\sigma}(g_0).
    $$
Thus, we have $(1-\alpha)^2=(1-\alpha)$ which contradicts the fact that $\alpha\in (0,1)$.
 \end{proof}

\section{Global unbalancedness for ergodic transformations does not hold} \label{endgame}

We will need the following lemma.

\begin{lemma} \label{notunbalanced}
  Let $G$ be a Polish group acting continuously on a Polish space. Suppose there is a comeager G-invariant set $C$ such that $C/G$ has meager connected components. Then the action of $G$ is not unbalanced.
\end{lemma}

\begin{proof}
    Suppose the action is unbalanced. Then there is a comeager $D$ in $C$ such that for all $x,y\in D$ there is a path between $Gx$ and $Gy$ in $C/G$. This implies that $D$ is contained in a single connected component of the $C/G$. This contradicts our assumption.
\end{proof}

Denote by ${\rm EMPT}(X,\mu)$ the set of ergodic transformations in ${\rm Aut}(X,\mu)$. The following theorem together with Lemma \ref{notunbalanced} implies Theorem \ref{Mainthm}.

\begin{theorem} \label{meagercomp}
    There is a conjugacy-invariant comeager $C\subset {\rm EMPT}(X,\mu)$ such that $C/{\rm Aut}(X,\mu)$ has meager connected components .
\end{theorem}
\begin{proof}
    Fix $\alpha\in (0,1)$. Let $C$ be the set of $\alpha$-weak mixing transformations. $C$ is comeager in ${\rm EMPT}(X,\mu)$ by Theorem \ref{stepin}.

    Take any $T$ in $C$. Denote by $D$ the connected component of $T$ in $C/{\rm Aut}(X,\mu)$. We will prove that $D$ is meager.
    
    Let $W$ be the set of all transformations which are disjoint with $T$. By \cite[Theorem 1]{Disjointgeneric}, $W$ is comeager. We will show that $D\cap W=\emptyset$. Take any $S\in W$. Suppose $T$ is connected with $S$ by path $T=T_1\leftrightsquigarrow...\leftrightsquigarrow T_n=S$  on $C/{\rm Aut}(X,\mu)$ (see Figure 1).

     \begin{figure}[h!]
\centering
\begin{tikzpicture}

    \draw[thick, fill=gray!30] (0,0) rectangle (10,6);

    \draw[black, ultra thick, fill=white, rounded corners=25pt] 
        (0.4, 0.4) -- (7.8, 0.4) -- (8.3, 3.0) -- (7.8, 5.6) -- (0.4, 5.6) -- cycle;
    
    \node[font=\sffamily, anchor=west] at (1.0, 4.5) {Transformations which are disjoint with $T$};

    \coordinate (T)      at (9.5, 5.5);  
    \coordinate (T2)     at (9.0, 4.75); 
    \coordinate (Tkm1)   at (9.0, 3.2);  
    \coordinate (Tk)     at (7.2, 2.7);  
    \coordinate (Tn_1)   at (4.5, 2.2);  
    \coordinate (S)      at (3.0, 2.1);  

    \draw[ultra thick, red] (T) -- (T2);
    
    \draw[ultra thick, red, dotted] (T2) -- (Tkm1) node[midway, right=3pt, text=black] {$\dots$};
    
    \draw[ultra thick, red] (Tkm1) -- (Tk);
    
    \draw[ultra thick, red, dotted] (Tk) -- (Tn_1) node[midway, above=2pt, text=black] {$\dots$};
    
    \draw[ultra thick, red] (Tn_1) -- (S);

    \filldraw[black] (T)      circle (2.5pt) node[anchor=east, xshift=-2pt] {$T$};
    \filldraw[black] (T2)     circle (2.5pt) node[anchor=east, xshift=-2pt] {$T_2$};
    \filldraw[black] (Tkm1)   circle (2.5pt) node[anchor=north west] {$T_{k-1}$};
    \filldraw[black] (Tk)     circle (2.5pt) node[anchor=south east] {$T_k$};
    \filldraw[black] (Tn_1)   circle (2.5pt) node[anchor=north] {$T_{n-1}$}; 
    \filldraw[black] (S)      circle (2.5pt) node[anchor=south, yshift=2pt] {$S$};
    
\end{tikzpicture}
\caption{}
\end{figure}

     Let $k$ be the least positive number such that $T_k$ and $T$ are disjoint. Since $T$ is disjoint with $S$, we know that such $k$ exists. By Theorem \ref{disjoint}, we know that $T_{k-1}$ and $T_{k}$ are not disjoint, thus $T_{k-1}\neq T$, in particular, $k>2$. By our choice of $k$, we have that $T_{k-1}$ and $T$ are not disjoint. Also $T_{k-1}$ and $T_k$ are not disjoint by Theorem \ref{disjoint}. Thus, $T_{k-1}$ is not disjoint with $T$ and $T_k$. On the other hand, $T$ and $T_k$ are disjoint by our choice of $k$. This contradicts Theorem \ref{genericdisjoint}, since $T_{k-1}$ is $\alpha$-weak mixing. This implies the theorem.
\end{proof}

\section{A positive result}
In this section, we discuss when for two ergodic transformations $T$ and $S$ we have $T\leftrightsquigarrow S$ and prove Theorem \ref{positive}. It is motivated by the fact that the proofs in Section \ref{shortcut} show that for any sequences $x,y\in\mathbb{R}^\omega$ there is a sequence $z\in\mathbb{R}^\omega$ such that $x\leftrightsquigarrow z$ and $z\leftrightsquigarrow x$.

 We will fix $(X,\mu)$ as a standard probability space. We heavily use the following fact.

 \begin{lemma} \label{Rokhlin}
     {\rm (Rokhlin's Lemma)} Let $T$ be an ergodic transformation on $X$, then for any $N\in \mathbb{N}$ and $\epsilon>0$, there is a measurable set $B$ such that
     \begin{enumerate}
         \item $B,TB,\ldots,T^{N-1}B$ are disjoint.
         \item $\mu(\bigcup_{k=0}^{N-1}T^kB)>1-\epsilon$.
     \end{enumerate}
 \end{lemma}
We call $\{T^kB\}_{0\leq k< N}$ an $(N,\epsilon)$-Rokhlin tower of $T$. The following stronger version is also well-known.
 \begin{lemma}\label{Rokhlinstrong}
      Let $T$ be an ergodic transformation on $X$, $\{A_1,\ldots,A_m\}$ be a partition of $X$ into measurable sets, then for any $N\in \mathbb{N}$ and $\epsilon>0$, there exists $M\in \mathbb{N}$ and disjoint measurable sets $\{B_{i}\}_{0\leq i<M}$ such that
      \begin{enumerate}
         \item The sets $\{T^kB_{i}\}_{0\leq k<N,0\leq i<M}$ are pairwise disjoint.
         \item $\mu(\bigcup_i\bigcup_{k=0}^{N-1}T^kB_i)>1-\epsilon$.
         \item For any $0\leq i<M$ and $0\leq k<N$, we have $T^kB_i$ is a subset of $A_j$ for some $0\leq j<m$.
     \end{enumerate}
 \end{lemma}
 \begin{proof}
     Put $M=m^N$. By Lemma \ref{Rokhlin}, we can find $B$ such that $T^k(B)$ are disjoint for $0\leq k<N$ and $\mu(\bigcup_{k=0}^{N-1}T^kB)>1-\epsilon$. Then for $\omega\in \{1,\ldots,m\}^N$, we define
     $$
    B_{\omega}=\{x\in B\mid\, \forall 0\leq i<N\,,T^i(x)\in A_{\omega(i)}\}.
     $$
     We have that $\{B_{\omega}\}_{\omega\in \{1,2,\dots,m\}^N}$ is a partition of $B$ and satisfies the lemma.
 \end{proof}
  We call $\{T^kB_{i}\}_{0\leq k<N,0\leq i<N}$ an \textbf{$(N,\epsilon)$-sequence of Rokhlin towers} of $T$ refining the partition $\{A_1,A_2,\dots,A_m\}$. The following lemma is also standard. 
  
  \begin{lemma} \label{towertotower}
      Let $T$ and $R$ be two ergodic transformations on $X$. Let $\{T^kB_i\}_{0\leq k<N, 0\leq i<M}$ and $\{H^kC_i\}_{0\leq k<N, 0\leq i<M}$ be $(N,\epsilon)$-sequences of Rokhlin towers of $T$ and $R$ refining two partitions, respectively. Suppose $\mu(B_i)=\mu(C_i)$ for all $0\leq i<M$, then there is $H\in {\rm Aut}(X,\mu)$ such that 
       \begin{enumerate}
           \item $HRH^{-1}=T$ on $\bigcup_{i=0}^{M-1}\bigcup_{k=0}^{N-2}T^kB_i$.
           \item $H(R^kC_i)=T^kB_i$ for all $0\leq k<N$ and $0\leq i <M$.
       \end{enumerate} 
  \end{lemma}
  \begin{proof}
      Since $\mu(B_i)=\mu(C_i)$, for all $0\leq i<M-1$ we can find measure preserving bijections $H_i: C_i\rightarrow B_i$. We now define $H=T^kH_i(R^{-1})^k$ on $R^kC_i$ for all $1\leq k< N$. Since $X\setminus\bigcup_{0\leq k<N-1,0\leq i<M}R^kC_i$ and $X\setminus\bigcup_{{0\leq k<N-1,0\leq i<M}}T^kB_i$ have the same measure, we can define $H$ to be any measure preserving bijection between them. This finishes the definition of $H$. Note that $H$ maps $R^kC_i$ to $T^kB_i$ for all $0\leq k<N$ and $0\leq i<M$. Take any $x\in T^kB_i$ with $0\leq k\leq N-2$, we may assume that $x=T^kH_ix_0$ for $x_0\in C_i$. Since $x\in T^kB_i$, we have 
      $$
      HRH^{-1}x=HRR^kH_i^{-1}T^{-k}T^kH_ix_0=H_iR^{k+1}x_0.
      $$
      Since $R^{k+1}x_0\in R^{k+1}C_i$, we have
      $$
      HR^{k+1}x_0=T^{k+1}H_iR^{-k-1}R^{k+1}x_0=T^{k+1}H_ix_0=Tx.
      $$
      This finishes the proof.
  \end{proof}

 \begin{lemma}\label{commonfactor}
     Let $T$ and $R$ be two ergodic transformations on $X$. Suppose that $T$ factors onto $R$ by $\pi_T$. Then for any weak open neighborhood $V$ of the identity, we can find an isomorphic copy $T'$ of $T$ and a sequence $H_n\in V$ such that $H_nRH_n^{-1}$ converges to $T'$ in the uniform metric.
 \end{lemma}
    \begin{proof}
        We may assume that $V$ is of the form:
        $$
        V=\{H\in {\rm Aut}(X,\mu)\mid \,\forall 1\leq j\leq m,\,\mu(H(A_j)\Delta A_j)<\delta \}
        $$
        where $\{A_j\}_{1\leq j\leq m}$ is a measurable partition of $X$. Note that $V$ is symmetric. 
       
        Fix an $N$ with $\frac{1}{N}<\frac{\delta}{4}$. By Lemma 
        \ref{Rokhlinstrong}, there is a $(N,\frac{\delta}{4})$-sequence of Rokhlin towers $\{R^kB_{i}\}_{0\leq k<N,0\leq i<M}$ of $R$ refining the partition $\{A_1,A_2,\ldots,A_m\}$.

        Let $B_{i,k}=R^kB_i$ and put $C_{i,k}=\pi_T^{-1}(B_{i,k})$. Note that $C_{i,k}$ are pairwise disjoint. Since $\pi_T$ is a factor map, we know that for any $0\leq i<M, 0\leq k<N-1$, we have $TC_{i,k}=C_{i,k+1}$. In other words, $\{C_{i,k}\}_{0\leq k<N,0\leq i<M}$ is an $(N,\frac{\delta}{4})$-sequence of Rokhlin towers of $T$ refining the partition $\{\pi_T^{-1}(A_j)\}_{j=1,\ldots,m}$. Let $C$ be the union of $X\setminus \bigcup_{0\leq i<M,0\leq k<N}B_{i,k}$ and $\bigcup_{0\leq i<M}B_{i,N-1}$. We have that $\mu(C)<\frac{1}{N}+\frac{\delta}{4}<\frac{\delta}{2}$.
        
        We take $G$ to be any measure preserving bijection which maps $C_{i,k}$ to $B_{i,k}$ and $T'=GTG^{-1}$. Now, $\pi_{T'}=\pi_T\circ G^{-1}$ is a factor map from $T'$ to $R$.
        
        \begin{claim} \label{inverse}
            For any $0\leq i<n, 0\leq k<N$, we have 
            $$
           \pi_{T'}^{-1}(B_{i,k})=B_{i,k}
            $$
        \end{claim}
        \begin{claimproof}
            Note that 
            $$
            x\in \pi_{T'}^{-1}B_{i,k} \Leftrightarrow \pi_{T'}(x)\in B_{i,k}\Leftrightarrow \pi_T\circ G^{-1}x\in B_{i,k}
            $$
            $$
            \Leftrightarrow G^{-1}x\in \pi_T^{-1}B_{i,k}=C_{i,k}
            \Leftrightarrow x\in GC_{i,k}=B_{i,k}.
            $$
        \end{claimproof}
  
        We will now prove the following: for every $0<\epsilon<\frac{\delta}{2}$, there is an $H\in V$ such that $d_u(HRH^{-1},T')<\epsilon$. Clearly, this is enough for the lemma. 
        
        Fix an $0<\epsilon<\delta$. Let $\mathcal{P}$ be the partition consisting of the sets $B_{i,k}$ and $X\setminus\bigcup_{i,k}B_{i,k}$. For $M_1\in \mathbb{N}$ such that $\frac{1}{M_1}<\frac{\epsilon}{2}$, we build up an $(M_1,\frac{\epsilon}{2})$-sequence of Rokhlin towers of $R$ refining the partition $\mathcal{P}$. Denote this sequence of towers by $\{D_{s,r}\}_{0\leq s< N_1, 0\leq r<M_1}.$ Note that,  the sets $\pi_{T'}^{-1}(D_{s,r})$ form an $(M_1,\frac{\epsilon}{2})$-sequence of Rokhlin towers of $T'$ refining the partition $\mathcal{P}$, by Claim \ref{inverse}.
        
         Let $D$ be the union of $X\setminus \bigcup_{0\leq s< N_1, 0\leq r<M_1}D_{s,r}$ and $\bigcup_{0\leq s<N_1} D_{s,M_1-1}$. Since for any $s,r$, we have $\mu(D_{s,r})=\mu(\pi_{T'}^{-1}D_{s,r})$, by Lemma \ref{towertotower}, we can find $H\in {\rm Aut}(X,\mu)$ such that $H(D_{s,r})=\pi_{T'}^{-1}D_{s,r}$ for any $0\leq s<n_1,0\leq r<M$ and $HRH^{-1}=T'$ except on $D$. Note that $\mu(D)<\frac{1}{M_1}+\frac{\epsilon}{2}<\epsilon$. This implies that $d_u(T',HRH^{-1})<\epsilon$.

        Now we show that $H\in V$. Note that every $D_{r,s}$ is a subset of $B_{i,k}$ for some $i,k$. For any $0\leq i<n$ and $0\leq k<N$, by Claim \ref{inverse}, we have $\pi_{T'}^{-1}(B_{i,k})=B_{i,k}$. Thus, $\pi_{T'}^{-1}D_{r,s}\subset \pi^{-1}_{T'} (B_{i,k})= B_{i,k}$ for the same $i,k$ such that $D_{r,s}\subset B_{i,k}$. In a conclusion, $H(B_{i,k})\subset B_{i,k}$ except on $C\cup (X\setminus \bigcup_{0\leq s< N_1, 0\leq r<M_1}D_{s,r})$  which has measure at most $\frac{\epsilon}{2}+\frac{\delta}{2}<\delta$. This finishes the proof.
    \end{proof}

 \begin{theorem} \label{commonsquig}
     Let $T$ and $S$ be ergodic transformations on $X$. If $T$ and $S$ share a common factor then $T\leftrightsquigarrow S$.
 \end{theorem}
 \begin{proof}
     Denote by $R$ a common factor of $T$ and $S$. Let $V$ be an open neighborhood of the identity in the weak topology, and $U$ be any open set in ${\rm Aut}(X,\mu)$.  
     
     Note that $R$ is a factor of ergodic transformation thus $R$ is ergodic hence almost everywhere aperiodic. Now, by Halmos' conjugacy Lemma \cite[Theorem 2.4]{Kechrisergodic}, the conjugacy class of $R$ is dense in the weak topology of ${\rm Aut}(X,\mu)$. Therefore, we can find an isomorphic copy of $R$, say $R'\in U$. Take a symmetric open neighborhood of the identity $W\subset V$, with the property that $W^2\subset V$ and 
     \begin{equation} \label{inU}
         \overline{\{HR'H^{-1}\mid\, H\in W\}}\subset U.
     \end{equation}
      
     By Lemma \ref{commonfactor}, we can find $T'$ isomorphic to $T$ and $S'$ isomorphic to $S$ such that there is a sequence of $H_n\in W$ and $F_n\in W$ with 
     $$
     d_u(H_nR'H_n^{-1},T')\rightarrow 0
     $$
     and
     $$
     d_u(F_nR'F_n^{-1},S') \rightarrow 0.
     $$ 
     Note that by (\ref{inU}), we have $T',S'\in U$. Since the uniform metric is bi-invariant, we have 
     $$
     d_u(F_nH_n^{-1}T'H_nF_n^{-1},S')=d_u(H_n^{-1}T'H_n,F_n^{-1}S'F_n)\leq 
     $$ 
     $$
     \leq  d_u(R',H_n^{-1}T'H_n)+d_u(R',F_n^{-1}S'F_n) 
     $$
     which converges to $0$. Similarly, we have
     $d_u(H_nF_n^{-1}S'F_nH_n^{-1},T')$ converges to $0$. By (\ref{inU}), both $F_nH_n^{-1}$ and $H_nF_n^{-1}$ are in $V$. Since the uniform metric is finer than weak topology, we have finished the proof that $T\leftrightsquigarrow S$.
 \end{proof}


Now we prove Theorem \ref{positive}.
 \begin{proof}[Proof of Theorem \ref{positive}]
     Let $T$ and $S$ be ergodic measura-preserving transformations on $X$. Using \cite[Propostion 1.4]{introtojoining}, take any ergodic joining $\nu$ of $T$ and $S$, let $J$ be any ergodic transformation  which is isomorphic to $(X\times X,\nu, T\times S)$ on $X$. By Theorem \ref{commonsquig}, we have $T\leftrightsquigarrow J\leftrightsquigarrow S$.
 \end{proof}
 
\section{Acknowledgments}
The author thanks François Le Ma\^{i}tre for valuable comments on a preliminary draft of the paper and Aristotelis Panagiotopoulos for helpful discussions on unbalancedness. We thank the Referee for finding an error in the initial version of this paper. We are indebted to Philipp Kunde and Marcin Sabok for discussions and help with fixing mistakes in the paper. We also thank Assaf Shani and Anush Tserunyan for their encouragement.  AI was only used to generate Figure 1 in this paper.
\bibliographystyle{abbrv} 
\bibliography{bibliography} 

\end{document}